\documentclass[a4paper,12pt]{article}

\usepackage{amsthm,amsmath,amssymb} 

 \usepackage{mathtools}
\usepackage{url} 
\usepackage{cases} 
\usepackage{tikz}
\usepackage{tikz-cd}
\usepackage{hyperref}

\newtheorem{theorem}{Theorem}[section]

\newtheorem{corollary}[theorem]{Corollary}

\theoremstyle{definition} 
\newtheorem{definition}{Definition}[section] 
\newtheorem{remark}{Remark}[section] 
\newtheorem{example}{Example}[section]

\newcommand{\F}{\mathbb{F}}
\newcommand{\K}{\mathrm{K}}
\newcommand{\GL}{\mathrm{GL}}
\newcommand{\M}{\mathrm{M}}

\newcommand{\PG}{\text{PG}}
\newcommand{\SL}{\text{SL}}

\newcommand{\hull}{\text{Hull}}
\newcommand{\pf}{\text{Pf}}

\newcommand{\diag}{\text{diag}}

\newcommand{\Rad}{\text{Rad}}
\newcommand{\rad}{\text{Rad}}
\newcommand{\G}{\mathcal{G}}
\newcommand{\rank}{\operatorname{rank}}

\title{Hull Variation in  Symplectic Linear Codes}
\author{Kanat Abdukhalikov \\	
Department of Mathematical Sciences, \\
UAE University, PO Box 15551, Al Ain, UAE\\
Email: abdukhalik@uaeu.ac.ae \medskip  \\ 
Somi Gupta \\
Department of Mathematical Sciences, \\
UAE University, PO Box 15551, Al Ain, UAE\\
Email: somi.gupta@uaeu.ac.ae \medskip \\
Duy Ho \\
Department of Mathematical Sciences, \\
UAE University, PO Box 15551, Al Ain, UAE\\
Email: duyho@uaeu.ac.ae
}
\date{ }

\begin{document} 

\maketitle
\begin{abstract} 
We study the variation of hull dimensions of symplectic linear codes under equivalence transformations preserving the symplectic distance. For a symplectic linear code $C$, we introduce an even-valued invariant $\Delta(C)$, which depends only on the equivalence class of $C$. This invariant can be expressed as the generic rank of an associated space of alternating matrices and, equivalently, as twice a corresponding matroid matching number. If $C$ has dimension $k$, we prove that the minimum symplectic hull dimension among all codes equivalent to $C$ is $k-\Delta(C)$. 
Consequently, $C$ is equivalent to a symplectic LCD code if and only if $\Delta(C)=k$. We further show that every symplectic MDS code of even dimension is monomially equivalent to a symplectic LCD code.

For a symplectic self-orthogonal code $C$, we determine all attainable hull dimensions: for every integer $r$ with $0\le r\le \Delta(C)/2$, there exists a monomially equivalent code whose symplectic hull has dimension $k-2r$. We also refine a standard form for the generator matrix of a symplectic self-orthogonal code and derive the corresponding parity-check matrix and the Gram matrix of the dual code. In the self-dual case, we characterize $\Delta(C)$ through minors of an associated symmetric matrix. Furthermore, we relate hull variation to graph matchings, establish a bound for $\Delta(C)$, 
and give graph constructions producing symplectic self-dual MDS, optimal, and best-known codes.
\end{abstract} 

\section{Introduction}

Let $C$ be a linear code over a finite field. 
The hull of $C$ is the intersection of $C$ with its dual: $\hull(C)=C\cap C^\perp$. 
This notion was introduced by Assmus and Key \cite{key} in connection with the classification of finite projective planes. 
It generalizes the notions of self-orthogonal codes and linear complementary dual (LCD) codes. 
The hull of a self-orthogonal code $C$ is $C$, and the hull of an LCD code is zero. Linear codes with various hull dimensions have attracted considerable attention lately since they can influence the complexity of some algorithms \cite{Sendrier97, Sendrier2000} and play an important role in quantum coding theory. 
Euclidean LCD codes were first studied by Massey \cite{Massey}, who established the existence of asymptotically good LCD codes. 
Carlet and Guilley \cite{Carlet} investigated applications of binary LCD codes to protection against side-channel attacks and fault-injection attacks, and presented several constructions of such codes. 
A major development in the theory was obtained by Carlet et al.  \cite{Mesnager}, who proved that every $q$-ary linear code is monomially equivalent to a Euclidean LCD code when $q > 3$, and every $q^2$-ary linear code is equivalent to a Hermitian LCD code when $q>2$.  
Luo et al.  \cite{Ling} further showed that a $q^2$-ary linear code with Hermitian hull of dimension $m$ can be transformed into an equivalent code with a smaller hull dimension $\ell< m$ when $q>2$.  
The analogous result holds for Euclidean hulls of $q$-ary codes when $q>3$. 
More general questions concerning the variation of hull dimensions were considered in \cite{Chen}.
The variation of hull dimensions for vector rank-metric codes has also been studied in \cite{Ho, ho2026}.

Much less is known about hull variation for symplectic linear codes, i.e., linear codes equipped with the symplectic inner product and symplectic distance. Such codes arise naturally in the study of quantum error-correcting codes and entanglement-assisted quantum error-correcting codes, see  \cite{AH26, AshikhminKnill01, ball2021, ball2025, BallMorenoSimoens25, BallPopatia26,   CRSS97, Calderbank, Ezerman21, galindo2019, KKKS06, Rains99}.  
Symplectic linear codes have recently been investigated in \cite{li2024, XuDu21, XuDu22, li2024b}. 

In this paper, we study symplectic linear codes under $\GL$-equivalence, which preserves symplectic distance \cite{GLPllaha19, Pllaha}.
We show that hull variation under this equivalence is governed by the generic rank of a matrix space, a notion introduced by L\'{a}szl\'{o} Lov\'asz in connection with graph matchings and the matroid parity problem \cite{Lovasz79, Lovasz80, Lovasz89}.
To each symplectic linear code $C$, we associate an invariant $\Delta(C)$ of the equivalence class of $C$.
In Section \ref{HVPlinea}, we show that the invariant $\Delta(C)$ is always an even integer and is equal to the generic rank of an associated matrix space.
 The invariant $\Delta(C)$ can also be interpreted as twice an associated matroid matching number.

 We further show that if $C$ is a symplectic linear code of length $n$ and dimension $k$, then any code $C'$ equivalent to $C$ has symplectic hull dimension greater than or equal to $k-\Delta(C)$, and there is a code $C''$ equivalent to $C$ with symplectic hull dimension exactly $k-\Delta(C)$. In particular, the code $C$ is equivalent to an LCD code if and only if $\Delta(C)=k$. 
We also prove that every symplectic maximum distance separable (MDS) code of even dimension is monomially equivalent to an LCD code.

For symplectic self-orthogonal codes, we obtain a complete hull variation result. 
 If $C$ is symplectic self-orthogonal, then,
for every integer $r$ satisfying $0 \le r \le \Delta(C)/2$, there exists a code $C'$ monomially equivalent to $C$ with symplectic hull dimension $k-2r$. 
Consequently, every admissible hull dimension between $k$ and $k-\Delta(C)$, with the same parity as $k$, occurs within the monomial-equivalence class of $C$; see Section \ref{HVPselforthogonal}.

We also establish several structural results for symplectic self-orthogonal codes. In Section~\ref{standard}, we refine the standard form of a generator matrix for symplectic self-orthogonal codes given in \cite{nielsenbook}. Using this form, we determine a parity-check matrix for $C$ and the Gram matrix of the dual code $C^{\perp_s}$.

If $C$ is a symplectic self-dual  code, then its generator matrix can be written in the form $G=[I \,|\, A]$, where $I$ is the identity matrix and $A$ is symmetric \cite{Danielsen2006, Danielsen2009,Danielsen2010,glynn,Tonchev}. In this setting, we prove that $\Delta(C)/2$ is the largest integer $r$ for which $A$ has a nonzero $r\times r$ minor involving no diagonal entries; see Section \ref{HVPselfdual}.

Finally, in Section \ref{graph}, we investigate connections among symplectic self-dual codes and the matching number $\nu(\G)$ of an associated graph $\G$. 
If $\G_C$ is the graph associated with a code $C$, then $\Delta(C)\leq 2\nu(\G_C)$. 
We further study connections of bipartite graphs and trees to the hull variation.
We provide graph constructions yielding symplectic self-dual MDS codes, optimal codes, and codes attaining the best-known parameters listed in \cite{grassl}. 
We also present examples of graphs with the property $\Delta(C)= 2\nu(\G_C)$.

\section{Preliminaries}\label{pre}

Let $q=p^r$ be a prime power and $\F_q$  be a finite field with $q$ elements. We define a symplectic inner product on $ \F_q^{2n}$. Let $u = (x_1, \ldots, x_n \mid y_1, \ldots, y_n)$ and 
$v = (x'_1, \ldots, x'_n \mid y'_1, \ldots, y'_n)$ be two vectors in 
$\F_q^{2n}$. Throughout the paper, we use a vertical bar to separate the first 
$n$ entries from the last $n$ entries. 
\noindent The \textit{symplectic inner product} is defined by
\[
\langle u, v \rangle_s 
= \sum_{i=1}^n 
\begin{vmatrix}
x_i & y_i \\ 
x'_i & y'_i 
\end{vmatrix}
= \sum_{i=1}^n (x_i y'_i - x'_i y_i).
\]
Equivalently, one can define the product $\langle \cdot , \cdot \rangle_s$ in matrix form as
\[
\langle u, v \rangle_s 
= u \, J_n v^{T},
\quad \text{where } 
J_n = 
\begin{bmatrix}
0 & I_n \\[3pt]
- I_n & 0
\end{bmatrix} \text{ and } I_n \text{ is the }n\times n \text{ identity matrix}.\]
It can be easily seen that every vector in $\F_q^{2n}$ is orthogonal to itself, i.e., $\langle u, u \rangle_s =0$. The \emph{symplectic distance} between $u$ and $v$ is defined as
\[
d_s(u,v) := w_s(u - v),
\]
where the \emph{symplectic weight} $w_s(u)$ is given by
\[
w_s(u) := 
\# \bigl\{\, i \in \{1,\dots,n\} : (x_i, y_i) \neq (0,0) \,\bigr\}.
\]
Equivalently,
\[
d_s(u,v) 
= \# \bigl\{\, i \in \{1,\dots,n\} : (x_i, y_i) \neq (x'_i, y'_i) \,\bigr\}.
\]
A symplectic linear code $C$ is a $k$-dimensional subspace of $\F_q^{2n}$. We define the distance $d$ of a code $C$ to be the minimum among the non-zero symplectic weights of its codewords. We denote $C$ as a $[2n,k,d]_q^s$ code. The dual code $C^{\perp_s}$ of $C$ is 
\[
C^{\perp_s}=\{u\in \F_q^{2n} : \langle c,u\rangle_s=0 \text{ for all } c\in C\}.
\]
A code $C$ can be represented by a generator matrix $G \in \F_q^{k\times2n}$ or a parity-check matrix $H\in\F_q^{(2n-k)\times2n}$, for which the code is the image and the kernel, respectively.
We define the symplectic hull of a code $C$ by $\hull(C)=C\cap C^{\perp_s}$. The \textit{Gram matrix} of the code $C$ is given by 
\[\Omega=GJ_nG^T. \] It is well-known that $\dim(\hull(C))=k-\rank(\Omega)$ (see \cite[Proposition 2]{galindo2019}). Moreover, $\Omega$ is an alternating matrix, hence $\rank(\Omega)$ is always even.

If $\hull(C)=C$, then the code $C$ is called \textit{symplectic self-orthogonal}. If  $\hull(C)=C=C^{\perp_s}$, then $C$ is called \textit{symplectic self-dual}. If $\hull(C)=0$, then $C$ is called a \textit{symplectic linear complementary dual (LCD) code}. Since we consider only the symplectic inner product, we omit the word ``symplectic'' when referring to these codes. 
 Let $C$ be a $[2n,k]_q^s$ code. The Singleton bound is given by \[d\leq n-\left\lceil \frac{k}{2} \right\rceil +1.\] 
 If the bound is attained, then $C$ is called a symplectic maximum distance separable (MDS) code.
   Let $u = (x_1, \ldots, x_n \mid y_1, \ldots, y_n)\in\F_q^{2n}$. We can regroup the $i$-th and $(i+n)$-th coordinates of $u$, for all $i\in\{1,2,\dots,n\}$, to get the form
$$u = (x_1, y_1 \,|x_2,y_2\,|\ldots \,|x_n, y_n).$$
\begin{definition}[$\GL$-equivalence]
Let $C_1$ and $C_2$ be two $[2n,k]_q^s$ codes. 
We say that $C_1$ and $C_2$ are \emph{$\GL$-equivalent} if there exist, a permutation $\pi\in S_n$ and matrices $A_1,A_2,\dots,A_n\in\GL_2(\F_q)$ such that, for all $u\in C_1$, the vector 
\[v= ((x_{\pi(1)},y_{\pi(1)}) A_{\pi(1)}|\,(x_{\pi(2)},y_{\pi(2)})A_{\pi(2)}|\,\dots|\,(x_{\pi(n)},y_{\pi(n)}) A_{\pi(n)})\]
belongs to $C_2$.
\end{definition}
We can express this definition in terms of generator matrices. 
Let $G_1$ and $G_2$ be  generator matrices of $C_1$ and 
$C_2$, respectively.
The code $C_1$ is $\GL$-equivalent to $C_2$ if and only if there exist a permutation matrix $P$ associated to some $\pi\in S_n$ and a block matrix $A=\diag(A_1,A_2,\dots,A_n)$, where $A_1,\dots,A_n\in\GL_2(\F_q)$, such that
\[G_2=G_1 A (P\otimes I_2).\]

\begin{definition}[$\text{SL}$-equivalence]
Let $C_1$ and $C_2$ be two  $[2n,k]_q^s$ codes with generator matrices $G_1$ and $G_2$, respectively. They are said to be $\text{SL}$-\emph{equivalent} if  there exist a permutation matrix $P$  associated to $\pi\in S_n$ and a block matrix $A=\diag(A_1,A_2,\dots,A_n)$ such that $G_2=G_1 A (P\otimes I_2)$, where $A_1,\dots,A_n\in\text{SL}_2(\F_q)$.
\end{definition}
It is easy to check that the $\text{SL}$-equivalences do not change the symplectic inner product of vectors, see for example \cite{AH26, Pllaha}. 
Furthermore, any matrix $A\in\GL_2(\F_q)$ can be factored as $A=D B$, where $D=\diag(\text{det}(A), 1)$ and $B\in\text{SL}_2(\F_q)$. Therefore, when we want to discuss equivalent codes that change symplectic hulls, we need to discuss codes $G_2=G_1 M$, 
where $M=\diag(d_1,1, d_2,1,\dots,d_n,1)$ and $d_i=\det(A_i)\in\F^*_q$, for all $i\in\{1,2,\dots,n\}$.
\begin{remark}\label{GLequimonomial}
    The $\GL$-equivalence preserves the symplectic distance/weight as for any $u,v\in\F_q^{2n}$:
    \[
d_s(uM,vM) =d_s(u,v).
\]  However,  $\GL$-equivalence does not preserve the symplectic inner product of codewords. Under the action of matrix $M$, the inner product of $u$ and $v$ transforms as
\[
\langle uM, vM \rangle_s 
= \sum_{i=1}^n d_i(x_i y'_i - x'_i y_i).
\]
\end{remark}

 A \textit{monomial matrix} is a square matrix having exactly one nonzero element in each row and column. Let $\mathcal{M}_{2n}(\F_q)$ be the group of $2n\times 2n$ monomial matrices with entries from $\F_q$. 
\begin{definition}[Monomial equivalence]
    Let $C_1$ and $C_2$ be two $[2n,k]_q^s$
codes with generator matrices $G_1$ and $G_2$, respectively.  Then, $C_1$ and $C_2$
are \textit{monomially equivalent} if there exists a matrix $M \in \mathcal{M}_{2n}(\F_q)$, such that
$
G_2 = G_1 M.
$ 
\end{definition} 

\begin{remark}
    Note that, by Remark \ref{GLequimonomial}, the study of    $\GL$-equivalences that change the symplectic hull is the same as the study of monomial equivalences. In further discussion, we can restrict ourselves to the monomial equivalences. 
\end{remark}
\begin{remark}
    Notice that to change the symplectic inner product, we need $d_i\in\F_q^*$ not all equal to 1, for $i\in\{1,2,\dots,n\}$. Therefore, if $q=2$, then it is not possible to change the symplectic inner product of codewords.
\end{remark}
\begin{remark}
    Note that if the generator matrix is written as
\[G = \big[\, a_1 \, a_2 \dots a_n\, | \, b_1 \, b_2\dots  b_n \,\big]\] then we have to rearrange the monomial matrix $M$ as $M=\diag(d_1,d_2,\dots,d_n,1, \dots, 1)$. They are just two ways of writing the generator matrix.  
\end{remark}

\section{Hull variation for symplectic linear codes}\label{HVPlinea}
From now on, we assume that $q\geq3$. In this section, we discuss codes monomially equivalent to the code $C$ and their symplectic hulls. 
 
Let $C$ be a linear $[2n,k]_q^s$ code with a generator matrix $G$ formed by the column vectors $a_i,b_i$ in $ \F_q^{k\times 1}$ for $i\in\{1,2,\dots,n\}$ as follows
\begin{equation}\label{generator}
    G = \big[\, a_1 \, b_1 \, | \, a_2 \, b_2 |\, \dots \,|\,a_n \, b_n \,\big].
\end{equation}
 Let 
\[M=\diag(d_1, 1,d_2, 1\dots,d_n,1),\] and  $D=\diag(d_1,d_2,\dots,d_n)$, where $d_1,\dots,d_n\in \F_q^*$. 
The code $C_D$ generated by the matrix $G_D=GM$ is a code monomially equivalent to $C$. We are interested in computing the symplectic hull dimension of $C_D$. Let the Gram matrix of $C_D$ be
\begin{equation*}\Omega_D=G_DJ'G^T_D=\sum_{i=1}^{n}d_i(a_ib_i^T-b
_ia_i^T),
\end{equation*}
where $J'=\diag(J_2, J_2,\dots,J_2)$ and $J_2=\begin{bmatrix}
    0 & 1\\-1 &0
\end{bmatrix}$.   Define \begin{equation*}
\Delta(C)=\max_{d_1,d_2,\dots,d_n\in \F_q^*}\{\rank(\Omega_D)\}.
\end{equation*} 
The number $\Delta(C)$ is an invariant of the equivalence class of $C$. Also, as $\Omega_D$ is an alternating matrix, $\Delta(C)$ is even. 
We are interested in finding $\Delta(C)$. We will follow the concept of generic rank introduced by L. Lovász \cite{Lovasz89} in 1989. Let $\mathcal{A}$ be a linear subspace of $m\times n$ matrices over $\F_q$. Then the generic rank of $\mathcal{A}$, denoted by $gr(\mathcal{A})$, is the maximum rank of matrices in it. 

 Let $L_C$ be the linear subspace of $\F_q^{k\times k}$ generated by $k\times k$ matrices $(a_ib_i^T-b
_ia_i^T)$, for $i\in\{1,2,\dots,n\}$. Here $a_i,b_i$ are as in Equation \eqref{generator}. In other words,
\begin{equation*}
    L_C=\langle (a_ib_i^T-b
_ia_i^T) : 1\leq i \leq n\rangle.
\end{equation*}
Then
\[gr(L_C)=\max_{d_1,d_2,\dots,d_n\in \F_q}\{\rank(\Omega_D)\}.\]
The matrices 
$(a_ib_i^T-b
_ia_i^T)$, for $1\leq i \leq n$,  have rank two unless $a_i$ and $b_i$ are proportional vectors. Let $S=\{a_{i_1},b_{i_1}, a_{i_2},b_{i_2}, \dots, a_{i_s},b_{i_s}\}$ be a largest family of pairs $a_i$ and $b_i$ that are collectively linearly independent. Using the arguments from \cite[p. 92]{Lovasz89}, it follows that the matrix space spanned by $(a_{i_j}b_{i_j}^T-b
_{i_j}a_{i_j}^T)$, for $1\leq j\leq s$, has generic rank $2s$. We will see that $gr(L_C)$ is also $2s$.

Let $\{e_1,e_2,\dots,e_k\}$ be the canonical basis of $\F_q^{k\times 1}$. Let $A$ be the matrix associated with the change of basis that takes the first ${2s}$ canonical basis elements to the set $S$. We can now see that $S=\{Ae_1,Ae_2,\dots,Ae_{2s}\}$. Consider a matrix $\Omega_{D_s}$ as
\begin{equation*}
\Omega_{D_s}=\sum_{j=1}^s \big(a_{i_j} b_{i_j}^T-b_{i_j}a_{i_j}^T\big)=A\Big(\sum_{j=1}^{s}\big(e_{2j-1} e_{2j}^T- e_{2j} e_{2j-1}^T\big)\Big)A^T,
\end{equation*}
so that
\[A^{-1}\Omega_{D_s}(A^T)^{-1}=\sum_{j=1}^{s}(e_{2j-1} e_{2j}^T- e_{2j} e_{2j-1}^T).\]
The $k\times k$ matrix on the right clearly has rank $2s$. Hence, $\rank(\Omega_{D_s})=2s$, giving us $gr(L_C)\geq2s$. 

Let us consider a matrix 
\[X(x_1,\dots,x_n)=\sum_{i=1}^{n}x_i(a_ib_i^T-b
_ia_i^T),\]
 where $x_i$ are indeterminates for $i\in\{1,2,\dots,n\}$. 
 By \cite{cheung, Lovasz89},  we have $\rank(X(x_1,\dots,x_n))=2s$ over $\F_q(x_1,x_2,\dots,x_n)$. 
 Let us consider \[X'(d_1,\dots,d_n)=\sum_{i=1}^{n}d_i(a_ib_i^T-b
_ia_i^T),\]  for any $d_i\in\F_q$. We claim that $\rank(X')\leq2s$. Suppose not, then there exists a $(2s+2)\times (2s+2)$ non-zero minor, say $D(d_1,\dots,d_n)$, for some  $d_i\in\F_q$. Now take the same minor in $X$, i.e., $D(x_1,\dots,x_n)$. This is not identically zero, as $D(d_1,\dots,d_n)$ is non-zero. Then, $\rank(X)>2s$, which is a contradiction. So $\rank(X')\leq 2s$, hence $gr(L_C)=2s$.

\begin{theorem}\label{delta=2s}
    Let $C$ be a linear  $[2n,k]_q^s$ code and $G\in\F_q^{k\times 2n}$ be the generator matrix as in \eqref{generator}. Then $\Delta(C)=gr(L_C)$.
\end{theorem}
    \begin{proof}
    As $gr(L_C)=2s$, by restricting $d_1,\dots,d_n\in\F^*_q$ we have $\Delta(C)\leq 2s$.
Let $S$ be a largest family of pairs $a_i$ and $b_i$ that are collectively linearly independent. Without loss of generality we can assume that $S=\{a_{1}, b_{1} , a_{2}, b_{2},  \dots ,a_{s}, b_{s}\}$. Consider the submatrix $G_S$ of $G$ as $$G_S=\big[\, a_{1} \, b_{1} \, | \, a_{2} \, b_{2} |\, \dots \,|\,a_{s} \, b_{s} \,\big]\in\F_q^{k\times 2s}.$$
Define the matrix $\Omega_{S}$ as
\[\Omega_{S}=\sum_{i=1}^s (a_{i} b_{i}^T-b_{i}a_{i}^T).\]  This  is a $(k\times k)$-matrix having rank $2s$. Now consider a $2s\times 2s$ submatrix $G_S'$ of $G_S$ such that the determinant of $G_S'$ is non-zero. 
Without loss of generality, we assume that $G_S'$ consists of the first $2s$ rows of $G_S$.  
Then \[G_S'=\big[\, a'_{1} \, b'_{1} \, | \, a'_{2} \, b'_{2} |\, \dots \,|\,a'_{s} \, b'_{s} \,\big]\in\F^{2s\times 2s}_q ,
\] where $a'_{i}$ and  $b'_{i}$ are restrictions of $a_{i}$ and $b_{i}$ to the first $2s$ coordinates. 
Now consider the matrix
\[B=G_S' J_s G_S'^T=\sum_{i=1}^{s}(a'_{i}b'^T_{i}-b'
_{i}a'^T_{i}).\]
Then $\rank(B)=2s$, as $G_S'$ and $J_s$ are invertible matrices. Let $x_1,\dots,x_n$ be indeterminates. Consider the following two matrices
\[X(x_1,\dots,x_n)=\sum_{i=1}^{n}x_i(a_ib_i^T-b
_ia_i^T) \,\,\, \text{ and }\,\,\, 
Y(x_1,\dots,x_n)=\sum_{i=1}^{n}x_i(a'_{i}b'^T_{i}-b'
_{i}a'^T_{i}),\]
where $a'_{i}$ and  $b'_{i}$ are restrictions of $a_{i}$ and $b_{i}$ to the first $2s$ coordinates for $i\in\{1,2,\dots,n\}$. The matrix $Y(x_1,\dots,x_n)$ is not identically zero as $B=Y(1,\dots,1, 0,\dots,0)$ and $\det(B)$ is non-zero. Moreover, $$\det(Y(x_1,\dots,x_n))=(\pf(Y(x_1,\dots,x_n)))^2.$$  Following the proof of \cite[Theorem 3]{Lovasz79}, the Pfaffian of $Y(x_1,\dots,x_n)$ is linear in every indeterminate $x_i$. Therefore, for $q>2$, there exists a tuple $(d_1,\dots,d_n)$, $d_i\in\F_q^*$, such that  $\det(Y(d_1,\dots,d_n))\neq0$. Since $Y(d_1,\dots,d_n)$ is a submatrix of $X(d_1,\dots,d_n)$ of rank $2s$, we have $\Delta(C)=2s=gr(L_C)$.
\end{proof}
\begin{corollary}
    \label{HVPlinear}
    Let $C$ be a $[2n,k]_q^s$  code and $G\in\F_q^{k\times 2n}$ be the generator matrix as in \eqref{generator}. Let  $C'$ be a code monomially equivalent to $C$.
    Then $\dim(\hull(C'))\geq k-\Delta(C)$. Moreover, there exists a code $C''$  monomially equivalent to $C$ such that $\dim(\hull(C''))=k-\Delta(C)$.
\end{corollary}

\begin{corollary}\label{LCD}
Let $C$ be a $[2n,k]_q^s$ code. Then $\Delta(C)=k$ if and only if $C$ is monomially equivalent to an LCD code.
\end{corollary}
Let us now consider MDS codes.

\begin{theorem}
     Let $C$ be an MDS $[2n,k]_q^s$ code and $G\in\F_q^{k\times 2n}$ be a generator matrix as in Equation \eqref{generator}. If $k$ is even, then $C$ is monomially equivalent to an LCD code.
\end{theorem}  
\begin{proof}
Let $m=\frac{k}{2} $.  We claim that if $C$ is an MDS code, then any collection of $m$ pairs
\[\{a_{i_1}, b_{i_1},a_{i_2}, b_{i_2},\dots, a_{i_{m}}, b_{i_{m}}\}\] is collectively linearly independent, for any $i_1,\dots,i_{m}\in\{1,2,\dots,n\}$. If not, there exist $j_1,\dots,j_{m}\in\{1,2,\dots,n\}$ such that  
$\{a_{j_1},b_{j_1}, a_{j_2},b_{j_2},\dots, a_{j_{m}}, b_{j_{m}}\}$
is not collectively linearly independent.  Then 
    \[G'=\big[a_{j_1} \, b_{j_1} | \, a_{j_2} \, b_{j_2} |\,\dots \,|\, a_{j_{m}}  b_{j_{m}}\big]\] 
    is a $k\times k$ singular matrix, i.e., there exists a non-zero $u\in\F_q^k$ such that 
    \[
    uG'=\big[\,ua_{j_1}\, ub_{j_1} \,|\, ua_{j_2}\, ub_{j_2} \,|\, \dots \,|\, ua_{j_m}\, ub_{j_m}\,\big]=0.
    \]
    Now consider  a codeword $c=uG$. Then it is clear that $c$ is non-zero and $w_s(c)\leq n-m$, contradicting the assumption that $C$ is MDS, i.e., $d=n+1-m$. Therefore, $\Delta(C)=k$ and the result follows from Corollary \ref{LCD}.
\end{proof} 

\begin{remark}\label{quantum}
    Let us see what $\Delta(C)$ means in projective terms. Consider the projective space $\PG(k-1,q)$ and the set $\mathcal{X}=\{\langle a_i, b_i\rangle:1\leq i \leq n\}\subseteq\PG(k-1,q)$, where $\langle a_i, b_i\rangle$ is a line if $a_i$ and $b_i$ are non-collinear vectors, and $\langle a_i, b_i\rangle$  is a point if $a_i$ and $b_i$ are collinear vectors. The set $\mathcal{X}$ generates $\PG(k-1,q)$. The number  $\Delta(C)/2$ is the maximum number of skew lines in $\mathcal{X}$. See also  \cite{ball2025}.
\end{remark}

\section{Hull variation for symplectic self-orthogonal codes}\label{HVPselforthogonal}

This section is devoted to the study of hull variation of a symplectic self-orthogonal code $C$ of dimension $k$. We will prove that it is always possible to find a monomially equivalent code $C'$ such that $\dim(\hull(C'))=k-2r$, for each $r\in\{0, 1, \dots,\Delta(C)/2\}$.

Let $C$ be a  self-orthogonal  $[2n,k]_q^s$ code with a generator matrix formed by the column vectors $a_i,b_i$ in $ \F_q^{k\times 1}$ for $i,j\in\{1,2,\dots,n\}$ as follows
\begin{equation*}
    G = \big[\, a_1 \, b_1 \, | \, a_2 \, b_2 |\, \dots \,|\,a_n \, b_n \,\big].
\end{equation*}
  Since the code $C$ is self-orthogonal, we have \begin{equation}\label{orthogonalGram}
    \Omega=GJ_nG^T=\sum_{i=1}^{n}(a_ib_i^T-b
_ia_i^T)=0. 
\end{equation}

Let $S$ be a largest family of pairs $a_i$ and $b_i$ that are collectively linearly independent. Without loss of generality, we can assume that $S=\{a_{1}, b_{1} , a_{2}, b_{2},  \dots ,a_{s}, b_{s}\}$. Now, select  $r\leq s$ pairs of $a_i, b_i$ from the set $S$, and define a matrix \[\Omega_{D_r}=(a_{1}b_{1}^T-b_{1}a_{1}^T)+\,(a_{2}b_{2}^T-b_{2}a_{2}^T)+\,\dots+(a_{r}b_{r}^T-b_{r}a_{r}^T)\] whose rank is $2r$. Choose $\gamma\in\F_q^*$ such that $1+\gamma\neq0$ and consider 
\[
\Omega_{D'_r}=\gamma\,\Omega_{D_r}+\Omega=\sum_{i=1}^{r}(1+\gamma)(a_ib_i^T-b_ia_i^T)+\sum_{i=r+1}^{n}(a_ib_i^T-b_ia_i^T),
\]
where $\Omega$ is as in \eqref{orthogonalGram}.
Therefore, $\Omega_{D'_r}\in L_C$ such that $\rank \Omega_{D'_r}=2r$.  
\begin{theorem}\label{hullvariationself}
     Let $C$ be a self-orthogonal $[2n,k,d]_q^s$ code with a generator matrix \[G = \big[\, a_1 \, b_1 \, | \, a_2 \, b_2 |\, \dots \,|\,a_n \, b_n \,\big].\]
     For every code $C'$ monomially equivalent to $C$, we have $\dim(\hull(C'))\geq k-\Delta(C).$
     Moreover, for every integer $0 \le r \le \Delta(C)/2$, there exists a $[2n, k,d]_q^s$ code $C_{D_r}$ monomially equivalent to $C$ such that
$$\dim(\hull(C_{D_r})) = k-2r.$$
\end{theorem}
\begin{proof}
 The first statement follows from Corollary \ref{HVPlinear}.  For $r\in\{0,1,\dots,\Delta(C)/2\}$, we see that $$\Omega_{D'_r}=\sum_{i=1}^{r}(1+\gamma)(a_ib_i^T-b_ia_i^T)+\sum_{i=r+1}^{n}(a_ib_i^T-b_ia_i^T)$$  has rank $2r$. Let $C_{D_r}$ be the code generated by $G_{D_r}=GM_r$  where 
  \[M_r=\diag(\underbrace{1+\gamma\,,1\,\ldots,1+\gamma,1}_{r},\underbrace{1,\,1,\ldots,1,\,1}_{n-r}).\]
  Then $C_{D_r}$ is monomially equivalent  to $C$ with hull dimension $k-2r$, for every $0\leq r\leq \Delta(C)/2$. 
\end{proof}
\begin{example}
Let $n$ be an even integer, and let $C$ be a self-dual $[2n,n,2]_q^s$ code with generator matrix $G=[I_n \mid \, A]$, where $A=\begin{bmatrix}
0 & I_{n/2} \\
I_{n/2} & 0
\end{bmatrix}$. Rewrite $G=\left[\begin{array}{cc|cc}
I_{n/2} & 0 &0 & I_{n/2}\\
0& I_{n/2} &  I_{n/2} & 0
\end{array}\right]$. Then $\Delta(C)=n$, since $S=\{e_i,e_{n/2+i}:1\leq i\leq n/2\}$. 
Define a matrix 
\[\Omega_{D_r}=\sum_{i=1}^{r}\left(e_{i}e_{i+n/2}^T-e_{i+n/2}e_{i}^T\right),\] whose rank is $2r$. Let
\[
\Omega=\sum_{i=1}^{n/2}\big(e_{i}e_{i+n/2}^T-e_{i+n/2}e_{i}^T\big)+\sum_{i=1}^{n/2}\big(e_{i+n/2}e_{i}^T-e_{i}e_{i+n/2}^T\big)=0.
\]
Choose $\gamma\in\F_q^*$, such that $1+\gamma\neq0$ and consider $\Omega_{D'_r}=\gamma\Omega_{D_r}+\Omega$. Then we get
\[\Omega_{D'_r}=\sum_{i=1}^{r}(1+\gamma)(e_{i}e_{i+n/2}^T-e_{i+n/2}e_{i}^T)+\sum_{i=r+1}^{n/2}(e_{i}e_{i+n/2}^T-e_{i+n/2}e_{i}^T)+\sum_{i=1}^{n/2}(e_{i+n/2}e_{i}^T-e_{i}e_{i+n/2}^T).\]
 Let $C_{D_r}$ be the code generated by $G_{D_r}=GM_r$  where 
  \[M_r= \diag(\underbrace{1+\gamma\,,1\,\ldots,1+\gamma,1}_{r},\underbrace{1,\,1,\ldots,1,\,1}_{n-r}).\]
  Then, $C_{D_r}$ is monomially equivalent  to $C$ with hull dimension $n-2r$, for every $0\leq r\leq n/2$.
\end{example}

\section{Standard form of self-orthogonal codes}\label{standard}
In this section, we provide a standard form for a self-orthogonal code $C$. We then determine the parity-check matrix of $C$ and the Gram matrix for $C^{\perp_s}$. We conclude this section by giving a structural description of $C^{\perp_s}$. Putting code into a standard form is a basic step towards developing a better understanding of the code.  
\begin{theorem}
Let $C$ be a symplectic self-orthogonal $[2n,k]_q^s$ code. Then $C$ is $\SL$-equivalent to a code $C'$ with generator matrix $G'$ in the form  
\begin{equation}\label{standardform}
G' = \big[\, I_k \; A_1 \mid (A_3 + A_1A_2^{T}) \; A_2 \big],
\end{equation}
where $A_1,A_2 \in \F_q^{k \times (n-k)},$ and
$A_3 \in \F_q^{k \times k}$  is symmetric.
\end{theorem}
\begin{proof}
 Let $C$ be a self-orthogonal $[2n,k]_q^s$ code. A form for the  generator matrix $G$ is given in \cite[p. 470]{nielsenbook} as follows

\begin{equation}
\label{nelson}G=\left[\begin{array}{c c c |c c c}
        I_{r}& A_{r\times (k-r)} & B_{r\times (n-k)} &  C_{r\times r} & 0 & D_{r\times (n-k)}\\
        0& 0 & 0 &  E_{(k-r)\times r} & I_{k-r} & F_{(k-r)\times (n-k)}    \end{array}\right].
\end{equation}

 In the Equation \eqref{nelson},  by swapping the columns $r+1,r+2,\dots,k$ with the columns $n+r+1,n+r+2,\dots,n+k$ respectively, and changing signs of columns $n+r+1,n+r+2,\dots,n+k$, we get
\begin{equation*}
G'=\left[\begin{array}{c c c |c c c}
        I_{r}& 0 & B_{r\times (n-k)} &  C_{r\times r} & -A_{r\times (k-r)}& D_{r\times (n-k)}\\
        0& I_{k-r} & 0 &  E_{(k-r)\times r} & 0 & F_{(k-r)\times (n-k)}    \end{array}\right].
\end{equation*}

Then $G'$ generates a code $C'$, $\SL$-equivalent to $C$. Therefore, by relabeling, we  get 
\begin{equation*}
    G'=[I_{k} \,\, A_1\,|\, A' \,\, A_2 ],
\end{equation*}
where $A_1\in\F_q^{_{k\times (n-k)}},A'\in\F_q^{k\times k},A_2\in\F_q^{k\times (n-k)}$. But as $C'$ is a self-orthogonal code, we know that the Gram matrix $\Omega=G'J_nG'^T=0$, i.e., $(A'-A_1A_2^T)^T=A'-A_1A_2^T.$ Let $A_3\coloneq A'-A_1A_2^T$. Then $A_3$ is a symmetric matrix and $A'=A_3+A_1A_2^T$. So we finally get a matrix as follows 
\begin{equation*}
    G'=[I_{k} \,\, A_1\,|\, (A_3+A_1A_2^T) \,\, A_2],
\end{equation*}
where $A_1\in\F_q^{_{k\times (n-k)}}, A_2\in\F_q^{k\times (n-k)}$ and $A_3$ is a symmetric matrix in $\F_q^{k\times k}$.
\end{proof}
\begin{definition}
    Let $C$ be a self-orthogonal $[2n,k]_q^s$ code.
    We say a generator matrix of the code $C$ is in \textit{standard form} if \begin{equation*}
    G=[I_{k} \,\, A_1\,|\, (A_3+A_1A_2^T) \,\, A_2],
\end{equation*}
where $A_1,A_2 \in \F_q^{k \times (n-k)}$,  and
$A_3 \in \F_q^{k \times k}$  is symmetric.
\end{definition}
\begin{theorem}\label{parity}    Let $C$ be a self-orthogonal $[2n,k]_q^s$ code whose generator matrix is in standard form \[
G = \big[\, I_k \; A_1 \mid (A_3 + A_1A_2^{T}) \; A_2 \big],
\] where $A_1,A_2 \in \F_q^{k \times (n-k)},$ and
$A_3 \in \F_q^{k \times k}$  is symmetric. Then a parity check matrix $H$ of $C$ (generator matrix of $C^{\perp_s}$) is given by
\begin{equation*}
H=
\left[
\begin{array}{cc|cc}
I_k & A_1 & A_3+A_1A_2^T & A_2 \\
0_{(n-k)\times k} & I_{n-k} & A_2^T & 0_{(n-k)\times(n-k)} \\
0_{(n-k)\times k} & 0_{(n-k)\times(n-k)} & -A_1^T & I_{n-k}
\end{array}
\right],
\end{equation*}
and the Gram matrix of $C^{\perp_s}$ is
\begin{equation*}
HJH^T=
\left[
\begin{array}{c c c}
0_{k\times k} & 0_{k\times (n-k)} & 0_{k \times (n-k)} \\
0_{(n-k)\times k} & 0_{(n-k)\times (n-k)}& I_{n-k}\\
0_{(n-k)\times k} & -I_{n-k}& 0_{(n-k)\times (n-k)}
\end{array}
\right].
\end{equation*}
\end{theorem}
\begin{proof}
    First note that all $(2n-k)$ rows of $H$ are linearly independent due to the presence of $I_{k}$, $I_{n-k}$, and $I_{n-k}$ in different rows and on different column blocks. Therefore, the rank of $H$ is exactly $2n-k$. We next see that $GJH^T=0$ and hence $H$ is a parity check matrix for $C$. Calculation of $HJH^T$ is straightforward.
\end{proof}
\noindent Let $V$ be  an $m$-dimensional vector space over $\F_q$. Let $\beta$ be a non-degenerate alternating bilinear form. Then it is well known that $\dim(V)=m=2n$ and $V$ can be written as an orthogonal direct sum of hyperbolic lines. This means that, if $\{e_1,e_2,\dots,e_n,f_1,f_2,\dots,f_n\}$ is a symplectic basis of $V$, i.e., $\beta(e_i,e_j)=0$, $\beta(f_i,f_j)=0$ for all $i,j$, and $\beta(e_i,f_j)=\delta_{ij}$, where $i,j\in\{1,2,\dots,n\}$ and $\delta_{ij}=\begin{cases}
0, & i\neq j  \\
1, & i=j
\end{cases}$,
then  $$V=\langle e_1, f_1\rangle\oplus\langle e_2, f_2\rangle\oplus\dots\oplus\langle e_n, f_n\rangle.$$ 
If $U$ is a subset of a vector space $V$, then we define $U^{\perp}=\{v\in V : \beta(u,v)=0, \text{for all } u\in U\}$. If $U$ is a subspace of $V$, then $\beta|_U:U\times U \rightarrow \F_q$ is an alternating form on $U$ and the radical of $U$ is $\rad(U)=U\cap U^{\perp}$. 

Using the description of the Gram matrix of $C^{\perp_s}$ from Theorem \ref{parity}, we can see the following. Let $c_i$ be the $i$-th row of  the matrix $H$ for $1\leq i\leq n$ and $d_{k+j-n}$ be the $j$-th row of $H$, where $n+1\leq j\leq 2n-k$. Then, $C^{\perp_s}$ has a nice description $$C^{\perp_s} = \langle c_1,c_2,\dots, c_k\rangle\oplus \langle c_{k+1},d_{k+1}\rangle\oplus\dots\oplus\langle c_{n},d_{n}\rangle.$$ 
If we consider $C^{\perp_s}$ as a vector space with symplectic inner product obtained by restricting the original symplectic inner product of $V$ to $C^{\perp_s}$, then $$\Rad(C^{\perp_s})=\langle c_1,c_2,\dots, c_k\rangle=C\cap C^{\perp_s}=\hull(C).$$  Since $C$ is self-orthogonal,  we have $\hull(C)=C$, hence $$C^{\perp_s} = C \oplus \langle c_{k+1},d_{k+1}\rangle\oplus\dots\oplus\langle c_{n},d_{n}\rangle.$$

\section{Hull variation for symplectic self-dual  codes}\label{HVPselfdual}
In this section, we restrict to the self-dual case. 
 Let $C$ be a self-dual $[2n,n]_q^s$  code. As $k=n$, a standard form of  generator matrix from Equation \eqref{standardform} becomes 
 \[G=[I\,|\,A],\] where $A\in\F_q^{n\times n}$ is a symmetric matrix. Let $M=\diag(d_1,d_2,\dots,d_n,1,1,\dots,1)$ and  $D=\diag(d_1,d_2,\dots,d_n)$, where $d_1,\dots,d_n\in \F_q^*$. 
 Note that we have rewritten the matrix $M$ according to the rearrangement of the columns of the generator matrix $G$.
 
 The code $C_D$ generated by the matrix $G_D=GM=[D\,|\, A]$ gives a code monomially equivalent to $C$. As we did previously, we are interested in computing the symplectic hull dimension of $C_D$. Let 
\begin{equation}\label{omegaself}
    \Omega_D=G_DJG^T_D=[D \, A] \begin{bmatrix}
0 & I_n \\
-I_n & 0
\end{bmatrix}\begin{bmatrix}
D \\
A
\end{bmatrix} = DA-AD.
\end{equation}
We see that $\Omega_D=[(d_i-d_j)a_{ij}]$ is an $(n\times n)$ alternating matrix, hence $\rank(\Omega_D)$ is always even and $\dim(\hull(C_D))=n-\rank(\Omega_D)$. We can easily see that \[\rank(\Omega_D)\leq \min\{n,2\rank(A)\}.\] 
We now use our Theorem \ref{hullvariationself} to give the following result.
\begin{corollary}
    Let $C$ be a self-dual $[2n,n]_q^s$ code.  For every code $C'$, monomially equivalent to $C$, we have $\dim(\hull(C'))\geq n-\Delta(C).$
     Moreover, for every integer $0 \le r \le \Delta(C)/2$, there exists a monomially equivalent $[2n, n]_q^s$ code $C_{D_r}$ such that
$$\dim(Hull(C_{D_r})) = n-2r.$$
\end{corollary}     
Generally, it is not easy to find $\Delta(C)$, but since the generator matrix of a self-dual code is well structured, we can further translate our problem into the existence of particular minors of $A$. This case was considered in \cite[Theorem 4.5]{li2024}. We improve their statement as follows. 

\begin{theorem}\label{minorhull}
Let $C$ be a self-dual $[2n, n]_q^s$ code with generator matrix $G=[I_n \mid A]$, where $A\in \mathbb{F}_q^{n\times n}$ is symmetric. There exists a code $C'$ monomially equivalent to $C$ with $\dim(\hull(C'))=n-2r$ if and only if there exists a non-zero $r\times r$ minor of $A$ not containing any diagonal entries of $A$.
\end{theorem} 
\begin{proof}
  Let $C'$ be a code monomially equivalent to $C$ such that $\dim(\hull(C'))=n-2r$, i.e., there exists a diagonal matrix $D'$ such that $\rank(\Omega_{D'})=2r$. Since the maximum of $\rank(\Omega_D)$ is $\Delta(C)$, there exists a set $S$ of cardinality $\Delta(C)$, containing pairs $e_i$ and $a_i$ of $G=[I\,|\,A]$, such that they are collectively linearly independent. Take $r$ pairs of $e_i,a_i$ from $S$, say
\[
e_{i_1}, e_{i_2},\dots,e_{i_r},a_{i_1},a_{i_2},\dots,a_{i_r}.
\]
Consider an $(n\times 2r)$-matrix, say $B$, containing the above vectors as columns. Then $\rank(B)=2r$. 
By deleting  the rows $i_1,i_2,\dots,i_r$, we obtain $B'$ as $$B'= [0_{(n-r)\times r} \,|\, A_{\{ j_1, j_2,\dots, j_{n-r}\};\{i_1,\dots,i_r\}}],$$  where $\{1,2,\dots,n\}\setminus\{i_1,\dots,i_r\}=\{j_1, j_2,\dots, j_{n-r}\}$.  
Therefore,  $\rank(B')=r$, i.e., there exists a non-zero $r\times r$ minor in $A_{\{ j_1, j_2,\dots, j_{n-r}\};\{i_1,\dots,i_r\}}$ not containing the diagonal entries of $A$. This is the required minor.

Conversely, assume that there exists a non-zero  $r\times r$ minor associated to a submatrix of $A$, say $M$, containing rows and columns indexed by $\mathcal{R}=\{i_1,i_2,\dots,i_r\}$ and $\mathcal{S}=\{j_1,j_2,\dots,j_r\}$  of $A$ respectively, i.e.,
    $$M=A_{\{i_1,i_2\dots,i_r\};\{j_1,j_2,\dots,j_r\}},$$
   such that $\mathcal{R}\cap \mathcal{S}=\varnothing$, as we want to avoid diagonal entries.
   Let $\sigma$ be a permutation  that sends the set $\mathcal{R}$ to the set $\{1,2,\dots,r\}$.  Let $P$ be the permutation matrix associated with $\sigma$ and $G'=[I\,| \,A']$ where $A'=PAP^T$, giving a code monomially equivalent to $C$. 
   Let $A'_{r,\,n-r}$ denote the submatrix of $A'$ formed by the first $r$ rows and the last $n-r$ columns. The previous operation permutes rows and columns of matrix $A$ such that $M$ becomes an $(r\times r)$-submatrix of $A'_{r,n-r}$. This is due to the fact that if $\mathcal{S}\cap \mathcal{R}=\varnothing$ then $\sigma (\mathcal{S})\cap \sigma(\mathcal{R})=\varnothing$. Then $\rank(A'_{r,n-r})=r$. By \cite[Theorem 4.5]{li2024}, there exists a code $C'$ monomially equivalent to $C$  with symplectic hull of dimension $n-2r$.\end{proof}

\begin{corollary}\label{minormax}
   Let $C$ be a self-dual $[2n, n]_q^s$ code with generator matrix $G=[I \,|\, A]$.  If $r$ is the largest integer such that there exists a non-zero  $r\times r$ minor of $A$ not containing the diagonal entries, then $\Delta(C)=2r$.
\end{corollary}
    \begin{corollary}
    Let $C$ be a symplectic self-dual  $[2n, n]_q^s$ code with generator matrix $G=[I \,|\, A]$. There exists a code $C'$ monomially equivalent to $C$ with $\dim(\hull(C'))< n$  if and only if  $A$ is not a diagonal matrix.
\end{corollary}

 Now we consider the case when $A=\diag(a_{11},a_{22},\dots,a_{nn})$ is a diagonal matrix. Let $e_1,\dots,e_n,f_1,\dots,f_n$ be a symplectic basis of $\F_q^{2n}$, in which the matrix $G=[ I | A]$ is defined. Then 
   $$C=\langle e_1+a_{11}f_1\rangle\oplus\cdots\oplus\langle   e_n+a_{nn}f_n\rangle.$$
We choose $n$ matrices in $\GL_2(\mathbb{F}_q)$  transforming $(1, a_{ii})$ to $(1,0)$, for all $1\leq i\leq n$. Then  $C$ is $\GL$-equivalent to the code  with generator matrix $G=[\,I_n \mid 0\,]$. In other words, $C$ is $\GL$-equivalent to the direct sum of orthogonal codes
$$C'=\langle e_1\rangle\oplus\cdots\oplus\langle   e_n\rangle.$$

  \section{Graphs associated to self-dual codes}\label{graph}
In this section, we establish connections between self-dual codes and graphs.  We give some examples of  optimal and best-known codes according to Grassl's table (see \cite{grassl}).

A graph is an ordered pair $\G = (V,E)$, where $V=\{v_1,v_2,\dots,v_n\}$  is a set whose elements are called vertices, and $E$ is a set of unordered pairs $\{v_i,v_j\}$, whose elements are called edges.  A graph with $n$ vertices can be represented by an
$n\times n$ adjacency matrix $A=(a_{ij})\in\M(n,\F_q)$, where \[a_{ij}=
\begin{cases}
1, & \{v_i,v_j\}\in E,\\
0, &\text{ otherwise.}
\end{cases}
\] 
Notice that we consider $a_{ij}$ to be from $\F_q$ and not from the real field as usually considered in graph theory.
A loop-free
graph has a symmetric adjacency matrix where all diagonal
elements are 0. A matching of a graph $\G$ is a set of edges without common vertices. If all the vertices are covered by disjoint edges (no common vertices), then we say $\G$ has a \textit{perfect matching}. A maximum matching is a matching that contains the largest possible number of edges. There may be many maximum matchings. The matching number $\nu(\G)$ of a graph $\G$ is the size of a maximum matching. So it is easy to see that if $\G$ has a perfect matching then $\nu(\G)=n/2$. 

Let $\G = (V, E)$ be a graph. Let $V_1$ and $V_2$ be two disjoint subsets of $V$ such that $V=V_1\cup V_2$. If 
$$E\subseteq\{(v_i,v_j): v_i\in V_1 \text{ and }v_j\in V_2\},$$
then $\G$ is said to be a \textit{bipartite graph}. If $|E|=|V_1||V_2|$, then we call $\G$ a complete bipartite graph and denote it by $\K_{|V_1|,|V_2|}$. It is well known that a graph is bipartite if and only if it has no odd cycle.  If a graph does not contain any cycle, then it is called a \textit{tree}.

Let us recall Tutte's matrix, used to determine the existence of a perfect matching. Let $\G$ be a graph and $V=\{v_1,v_2,\dots,v_n\}$ be the set of vertices. The Tutte matrix $T_{\G}$ is defined in \cite{Tutte} as
\[{T_\G}=\sum_{(v_i,v_j)\in E, i<j} x_e(E_{i,j
}-E_{j,i}),
\]
where $x_e$ is an  indeterminate indexed by $e=(v_i,v_j)\in E$. The coefficient matrix  $E_{i,j
}-E_{j,i}$ is the rank-two alternating matrix whose $(i,j)$-entry is $+1$, $(j,i)$-entry is $-1$ and others are 0, for $i<j$. It is well known that $2\nu(\G)=\rank(T_{\G})$. Therefore, $T_\G$ is non-singular if and only if $\G$ has a perfect matching. 

Let $C$ be a $[2n,n]_q^s$ self-dual code with generator matrix $G=[I \, |\,A]$, where $A$ is a symmetric matrix over $\F_q$. Let us associate a graph $\G_{C}$ to the code $C$ by taking $A'=(a'_{ij})$ as the adjacency matrix of $\G_C$, where
$$a'_{ij}=\begin{cases}
0, & a_{ij}=0,  \\
1, & a_{ij}\neq0.
\end{cases}$$ Note that the mapping $C\rightarrow\G_C$ is non-injective, i.e., many codes might be mapped to the same graph $\G_C$.

 \begin{theorem}\label{delatalessmatching}
     Let $C$ be a self-dual $[2n,n]_q^s$ code with generator matrix $G=[I\,|\,A]$. Then $\Delta(C)\leq 2\nu(\G_C)$.
 \end{theorem}
 \begin{proof}
  Let $\Omega_D$ be as defined in Equation \eqref{omegaself}. Then
  \[\Omega_D= \sum_{(i,j)\in E, i<j} a_{ij}(d_i-d_j)(E_{i,j}-E_{j,i}).
\]It is easy to  see that $\rank(\Omega_D)\leq\rank(T_\G)$ since $x_e$ are indeterminates, however $(d_i-d_j)a_{ij} \in\F_q$, for all $i,j\in\{1,2,\dots,n\}$ . Therefore, $\Delta(C)\leq \rank(T_\G)= 2\nu(\G_C)$.
\end{proof}

 \begin{example}
We provide an example to demonstrate that the inequality in Theorem~\ref{delatalessmatching} can be strict. Let $C$ be the self-dual $[12,6]_q^s$ code with generator matrix $G=[I_6\mid A]$, where
\[
A=\begin{bmatrix}
0&0&0&1&1&1\\
0&0&0&1&1&1\\
0&0&0&1&1&1\\
1&1&1&0&0&0\\
1&1&1&0&0&0\\
1&1&1&0&0&0
\end{bmatrix}.
\]
Then $\G_C$ is the complete bipartite graph $\K_{3,3}$, which has a perfect matching, i.e., $\nu(\G_C)=3$. Moreover, $\Delta(C)=4$, since
$a_1=(1,0,0,0,0,0)^T$, $b_1=(0,0,0,1,1,1)^T$, $a_6=(0,0,0,0,0,1)^T$, $b_6=(1,1,1,0,0,0)^T$
form a largest set of collectively linearly independent pairs of columns of $G$. Hence $\Delta(C)<2\nu(\G_C)$.
\end{example}

\begin{theorem}\label{rminormatching}
Let $C$ be a self-dual $[2n,n]_q^s$ code with generator matrix $G=[I_n\, |\, A ]$. Then the graph $\G_C$ has a matching of size $\Delta(C)/2$. 
\end{theorem}
\begin{proof}
   By Theorem \ref{minorhull}, there is a non-zero  $(\Delta(C)/2\times \Delta(C)/2)$-minor, associated to a submatrix $B$, not containing diagonal entries of $A$. Since $\rank(B)=\Delta(C)/2$, using this minor we can select a matching of size  $\Delta(C)/2$. 
  \end{proof}

Conversely, we can construct self-dual codes from graphs.    
Let $\G$ be a graph on $n$ vertices with adjacency matrix $A$. 
We denote by $C_{\G}$ the self-dual $[2n,n]_q^s$ code with generator matrix $[I\mid A]$. 
The graph associated with $C_{\G}$ is $\G$ itself, so Theorems~\ref{delatalessmatching} and~\ref{rminormatching} apply to $C_{\G}$ with $\G_{C}$ replaced by $\G$.
 
  \begin{theorem}
 Let $\G$ be a graph with adjacency matrix $A$. If $\{v_{i_1},v_{j_1}\},\dots, \{v_{i_r},v_{j_r}\}$ is a maximum matching of $\G$
and the minor $M_{\{i_1,i_2\dots,i_{r}\};\{j_1,j_2,\dots,j_{r}\}}$ of $A$ is non-zero, then $\Delta(C_{\G})=2r=2\nu(\G)$. 
\end{theorem} 
\begin{proof}
   Since there exists a non-zero minor $M_{\{i_1,i_2\dots,i_{r}\};\{j_1,j_2,\dots,j_{r}\}}$ of $A$,  by Theorem \ref{minorhull}, there exists a code $C'$ monomially equivalent to $C_{\G}$ with $\dim(\hull(C'))=n-2r$.
   By Theorem \ref{delatalessmatching}, $2r\leq \Delta(C_{\G})\leq2\nu(\G)=2r$, we have $\Delta(C_{\G})=2r$.
\end{proof}
Now we consider some special classes of graphs.
\begin{theorem}\label{bipartite}
      Let $\G$ be a bipartite graph with $n$ vertices and adjacency matrix $A$.  If $\G$ has a unique perfect matching then $\Delta(C_{\G})=n$.
      \end{theorem}
      \begin{proof}
      Since $\G$ is a bipartite graph and  has a unique perfect matching,  by \cite[Lemma 2.1]{godsil} we have that $\det(A)=\pm1$ and $n$ is even. Hence, we can rearrange the vertices to get  \[A'=PAP^T=\begin{bmatrix}
              0&B\\B^T &0
          \end{bmatrix},\] where $P$ is a permutation matrix and $B$ is an invertible $\frac{n}{2}\times \frac{n}{2}$ square  matrix.
        Then $G'=[I\,|\,A']$ is a generator matrix of a code $C'$ equivalent to $C_{\G}$. Since $\det(B)\neq0$, the $\frac{n}{2}$ columns of $B^T$ are linearly independent. Denote the first $n/2$ columns of $A'$ as $a'_1,a'_2,\dots,a'_{\frac{n}{2}}$.  Select the columns  \[e_1,a'_1, \,e_2,a'_2,\dots, e_{\frac{n}{2}}, a'_{\frac{n}{2}},\]
        from the matrix $G'$.
        They are clearly $\frac{n}{2}$ collectively linearly independent pairs. Hence $\Delta(C_{\G})=n$.
      \end{proof}

\begin{theorem}
 Let $\G$ be a tree with $n$ vertices and adjacency matrix $A$. Then the following statements are equivalent. 
 \begin{enumerate}
    
     \item The determinant of $A$  is non-zero.
     \item The tree $\G$ has a perfect matching.
     \item The tree $\G$ has a unique perfect matching.
     \item   The number $n$ is even and $ \Delta(C_{\G})=n$.
     \item The code $C_{\G}$ is monomially equivalent to an LCD code.
 \end{enumerate}
\end{theorem}
\begin{proof}
   It is known that the adjacency matrix of a tree is non-singular
if and only if the tree has a unique perfect matching. Since trees are bipartite graphs we have $\det(A)=\pm{1}$ (see \cite{godsil}). Therefore, the first three statements are equivalent. The last two statements are equivalent, since $\Delta(C_{\G})=n$ if and only if there exists a code $C'$ monomially equivalent to $C_{\G}$ such that $\dim(\hull(C'))=0$.
    
Assume that the tree $\G$ has a unique perfect matching; then by Theorem \ref{bipartite}, we have $\Delta(C_{\G})=n$. This proves that statement 3 implies statement 4. By Theorem \ref{rminormatching},  if  $\Delta(C_{\G})=n$ then there exists an $(n/2)$-matching, which is a perfect matching. This proves that statement 4 implies statement 2.
\end{proof}
We give some examples of graphs that yield MDS, optimal, and best-known self-dual codes.
In each of the following examples, $\Delta(C_{\G})=2\nu(\G)$, that is, the bound of Theorem~\ref{delatalessmatching} is attained.

\begin{example}
Let  $\G$ be the wheel on six vertices with adjacency matrix 
    \[
A=
\begin{pmatrix}
0 & 1 & 1 & 1 & 1 & 1 \\
1 & 0 & 1 & 0 & 0 & 1 \\
1 & 1 & 0 & 1 & 0 & 0 \\
1 & 0 & 1 & 0 & 1 & 0 \\
1 & 0 & 0 & 1 & 0 & 1 \\
1 & 1 & 0 & 0 & 1 & 0
\end{pmatrix}.
\] Then $C_{\G}$ is a $[12,6]_q^s$ code. Calculation by Magma shows that for $q\leq 1024$, the minimum symplectic distance of the code $C_{\G}$ is 4. Therefore, $C_{\G}$ is a symplectic MDS code for $q\leq1024$. Also, $\G$ has a perfect matching. Notice that 
\[e_1=(1,0,0,0,0,0),\,\, b_1=(0,1,1,1,1,1)\]
\[e_2=(0,1,0,0,0,0),\,\, b_2=(1,0,1,0,0,1)\]
\[e_3=(0,0,1,0,0,0),\,\, b_3=(1,1,0,1,0,0)\]
are three linearly independent pairs that are collectively independent; hence $\Delta(C_{\G})=6=2\nu(\G)$.
\end{example}
\begin{example}
    Let $\G$ be the Petersen graph with adjacency matrix
    \[
A=
\begin{pmatrix}
0&1&0&0&1&1&0&0&0&0\\
1&0&1&0&0&0&1&0&0&0\\
0&1&0&1&0&0&0&1&0&0\\
0&0&1&0&1&0&0&0&1&0\\
1&0&0&1&0&0&0&0&0&1\\
1&0&0&0&0&0&0&1&1&0\\
0&1&0&0&0&0&0&0&1&1\\
0&0&1&0&0&1&0&0&0&1\\
0&0&0&1&0&1&1&0&0&0\\
0&0&0&0&1&0&1&1&0&0
\end{pmatrix}
\] Then $C_{\G}$ is a $[20, 10]_q^s$ code. Using Magma, we calculate that the minimum symplectic distance of the code $C_{\G}$ is 4 for $q\leq1024$. It is well known that the Petersen graph $\G$ has a perfect matching, i.e., $\nu(\G)=5$. Notice that
\[e_1=(1,0,0,0,0,0,0,0,0,0),\,\, b_1=(0,1,0,0,1,1,0,0,0,0)\]
\[e_2=(0,1,0,0,0,0,0,0,0,0),\,\, b_2=(1,0,1,0,0,0,1,0,0,0)\]
\[e_3=(0,0,1,0,0,0,0,0,0,0),\,\, b_3=(0,1,0,1,0,0,0,1,0,0)\]
\[e_4=(0,0,0,1,0,0,0,0,0,0),\,\, b_4=(0,0,1,0,1,0,0,0,1,0)\]
\[e_5=(0,0,0,0,1,0,0,0,0,0),\,\, b_5=(1,0,0,1,0,0,0,0,0,1)\]
are 5 linearly independent pairs that are collectively independent; therefore, $\Delta(C_{\G})=10=2\nu(\G)$.
\end{example}

\begin{example}
   Let us consider a graph $\G$ with an adjacency matrix $A$ given by\[
a_{ij}=
\begin{cases}
1, & \text{if } j-i\equiv \pm1,\pm3,6 \pmod{12},\\
0, & \text{otherwise}.
\end{cases}
\]
The explicit representation of the matrix $A$ is \[
A=
\begin{pmatrix}
0&1&0&1&0&0&1&0&0&1&0&1\\
1&0&1&0&1&0&0&1&0&0&1&0\\
0&1&0&1&0&1&0&0&1&0&0&1\\
1&0&1&0&1&0&1&0&0&1&0&0\\
0&1&0&1&0&1&0&1&0&0&1&0\\
0&0&1&0&1&0&1&0&1&0&0&1\\
1&0&0&1&0&1&0&1&0&1&0&0\\
0&1&0&0&1&0&1&0&1&0&1&0\\
0&0&1&0&0&1&0&1&0&1&0&1\\
1&0&0&1&0&0&1&0&1&0&1&0\\
0&1&0&0&1&0&0&1&0&1&0&1\\
1&0&1&0&0&1&0&0&1&0&1&0
\end{pmatrix}.
\]

Then $C_{\G}$ is a $[24,12]_q^s$ code. Using Magma, we calculate that  the minimum symplectic distance of the code $C_{\G}$ is 6, for $q\leq 1024$. This type of code can be found in \cite[p. 51]{glynn} for $\F_4$, and we have generalized it to any $q$. For  $q=3$ and $4$, these are optimal codes according to \cite{grassl}. Consider
\[e_1=(1,0,0,0,0,0,0,0,0,0,0,0),\,\, a_1=(0,1,0,1,0,0,1,0,0,1,0,1)\]
\[e_2=(0,1,0,0,0,0,0,0,0,0,0,0),\,\, a_2=(1,0,1,0,1,0,0,1,0,0,1,0)\]
\[e_3=(0,0,1,0,0,0,0,0,0,0,0,0),\,\, a_3=(0,1,0,1,0,1,0,0,1,0,0,1)\]
\[e_4=(0,0,0,1,0,0,0,0,0,0,0,0),\,\, a_4=(1,0,1,0,1,0,1,0,0,1,0,0)\]
\[e_5=(0,0,0,0,1,0,0,0,0,0,0,0),\,\, a_5=(0,1,0,1,0,1,0,1,0,0,1,0)\]
\[e_7=(0,0,0,0,0,0,1,0,0,0,0,0),\,\, a_7=(1,0,0,1,0,1,0,1,0,1,0,0).\]
These are 6 linearly independent pairs that are collectively independent; therefore, $\Delta(C_{\G}) =
12=2\nu(\G)$.
\end{example}
\begin{example}
     Let $\G$ be a graph with adjacency matrix $A$ given by
     \[
a_{ij}=
\begin{cases}
1, & \text{if } j-i\equiv \pm1,\pm3,\pm6 \pmod{15},\\
0, & \text{otherwise}.
\end{cases}
\]
Explicitly,
\[
A=
\begin{pmatrix}
0&1&0&1&0&0&1&0&0&1&0&0&1&0&1\\
1&0&1&0&1&0&0&1&0&0&1&0&0&1&0\\
0&1&0&1&0&1&0&0&1&0&0&1&0&0&1\\
1&0&1&0&1&0&1&0&0&1&0&0&1&0&0\\
0&1&0&1&0&1&0&1&0&0&1&0&0&1&0\\
0&0&1&0&1&0&1&0&1&0&0&1&0&0&1\\
1&0&0&1&0&1&0&1&0&1&0&0&1&0&0\\
0&1&0&0&1&0&1&0&1&0&1&0&0&1&0\\
0&0&1&0&0&1&0&1&0&1&0&1&0&0&1\\
1&0&0&1&0&0&1&0&1&0&1&0&1&0&0\\
0&1&0&0&1&0&0&1&0&1&0&1&0&1&0\\
0&0&1&0&0&1&0&0&1&0&1&0&1&0&1\\
1&0&0&1&0&0&1&0&0&1&0&1&0&1&0\\
0&1&0&0&1&0&0&1&0&0&1&0&1&0&1\\
1&0&1&0&0&1&0&0&1&0&0&1&0&1&0
\end{pmatrix}.
\]
Then $C_{\G}$ is a $[30,15]_q^s$ code. Using Magma, we calculate that the minimum symplectic distance of code $C_{\G}$ is 6, for $q$ odd, and 5 for $q$ even, for any $q\leq1024$. For $q=3$, this is the best-known code as in \cite{grassl}. Now we see that 
\[e_1=(1,0,0,0,0,0,0,0,0,0,0,0,0,0,0),\,\,
a_1=(0,1,0,1,0,0,1,0,0,1,0,0,1,0,1)\]
\[e_2=(0,1,0,0,0,0,0,0,0,0,0,0,0,0,0),\,\,a_2=(1,0,1,0,1,0,0,1,0,0,1,0,0,1,0)\]
\[e_3=(0,0,1,0,0,0,0,0,0,0,0,0,0,0,0),\,\,a_3=(0,1,0,1,0,1,0,0,1,0,0,1,0,0,1)\]
\[e_4=(0,0,0,1,0,0,0,0,0,0,0,0,0,0,0),\,\,a_4=(1,0,1,0,1,0,1,0,0,1,0,0,1,0,0)\]
\[e_5=(0,0,0,0,1,0,0,0,0,0,0,0,0,0,0),\,\,a_5=(0,1,0,1,0,1,0,1,0,0,1,0,0,1,0)\]
\[e_7=(0,0,0,0,0,0,1,0,0,0,0,0,0,0,0),\,\,a_7=(1,0,0,1,0,1,0,1,0,1,0,0,1,0,0)\]
\[e_9=(0,0,0,0,0,0,0,0,1,0,0,0,0,0,0),\,\,
a_9=(0,0,1,0,0,1,0,1,0,1,0,1,0,0,1)\]
 are 7 pairs of linearly independent vectors that are collectively independent. Then $\Delta(C_{\G})=14 =2\nu(\G)$. Note that the hull dimension of any code equivalent to $C_{\G}$ cannot be 0, i.e., $C_{\G}$ is not equivalent to an LCD code, but we can find a monomially equivalent code $C'$ such that $\dim(\hull(C'))=1$. 
\end{example}

\section*{Acknowledgment}

This work was supported by UAEU (Grant G00005533).


\begin{thebibliography}{99} 

\bibitem{AH26}
Abdukhalikov, K., \& Ho, D. (2026). On the equivalence between additive and linear codes. IEEE Transactions on Information Theory. 
10.1109/TIT.2026.3726229.


\bibitem{AshikhminKnill01}
Ashikhmin, A., \& Knill, E. (2001). Nonbinary quantum stabilizer codes. IEEE Transactions on Information Theory, 47(7), 3065--3072.



\bibitem{key}
Assmus, E.F. Jr., \& Key, J. D. (1990).  Affine and projective planes. Discrete Mathematics, 83,  161–187. 

\bibitem{ball2021}
Ball, S. (2021). Some constructions of quantum MDS codes. Designs, Codes and Cryptography, 89(5), 811-821.



\bibitem{ball2025}
Ball, S., Centelles, A., \& Huber, F. (2023). Quantum error-correcting codes and their geometries. Ann. Inst. Henri Poincaré Comb. Phys. Interact. 10(2), 337–405.

\bibitem{BallMorenoSimoens25}
Ball, S., Moreno, E., \& Simoens, R. (2025). Stabiliser codes over fields of even order. IEEE Transactions on Information Theory, 71, 3707--3718.

\bibitem{BallPopatia26}
Ball, S., \& Popatia, T. (2026). Additive codes from linear codes. IEEE Transactions on Information Theory, 72, 342--345.



\bibitem{CRSS97}
Calderbank, A. R., Rains, E. M., Shor, P. W., \& Sloane, N. J. A. (1997). Quantum error correction and orthogonal geometry. Physical Review Letters, 78(3), 405--408.



\bibitem{Calderbank}
Calderbank, A. R.,  Rains,  E. M.,  Shor, P. W., \& Sloane, N. J. A. (1998). Quantum error correction via codes over GF(4). IEEE Trans. Inf. Theor. 44(4), 1369–1387. 
 

\bibitem{Carlet}
Carlet, C.  \& Guilley, S. (2016). Complementary dual codes for counter-measures to side-channel attacks. Advances in Mathematics of Communications, 10(1), 131--150.

\bibitem{Mesnager}
Carlet, C., Mesnager,  S., Tang, C., Qi, Y., \& Pellikaan, R. (2018). Linear codes over $\mathbb F_q$ are equivalent to LCD codes for $q>3$. IEEE Transactions on Information Theory, 64(4), 3010-3017. 


\bibitem{cheung}
Cheung, H.Y.,  Lau, L.C., \& Leung, K.M. (2014). Algebraic algorithms for linear matroid parity problems. ACM Transactions on Algorithms, 10(3), 1–26.

\bibitem{Chen}
Chen, H. (2023). On the hull-variation problem of equivalent linear codes. IEEE Transactions on Information Theory, 69(5), 2911-2922.

\bibitem{Danielsen2006}
Danielsen L. E., Parker M. G. (2006). 
On the classification of all self-dual additive codes over $GF(4)$ of length up to 12.  
Journal of Combinatorial Theory, Series A, 113, 1351--1367. 

\bibitem{Danielsen2009}
Danielsen, L.E. (2009). 
Graph-based classification of self-dual additive codes over finite fields. 
Advances in Mathematics of Communications, 3(4), 329-348. 

\bibitem{Danielsen2010}
Danielsen, L. E. (2012). 
On the classification of Hermitian self-dual additive codes over GF(9). 
IEEE Transactions on Information Theory, 58(8), 5500--5511.

\bibitem{Ezerman21}
Ezerman, M. F. (2021). Quantum error-control codes. In W. C. Huffman, J.-L. Kim, \& P. Sol\'e (Eds.), Concise Encyclopedia of Coding Theory (Chapter 27). Chapman and Hall/CRC.





\bibitem{galindo2019}
Galindo, C., Hernando, F., Matsumoto, R., \& Ruano, D. (2019). Entanglement-assisted quantum error-correcting codes over arbitrary finite fields. Quantum Information Processing, 18(4).

\bibitem{GLPllaha19}
Gluesing-Luerssen, H., \& Pllaha, T. (2019). On quantum stabilizer codes derived from local Frobenius rings. Finite Fields and Their Applications, 58, 145--173.

\bibitem{glynn}
 Glynn, D.G., Gulliver,  T.A., Maks, J.G.,  \&  Gupta, M.K.  (2014). The Geometry of
Additive Quantum Codes. unpublished manuscript  available online at
https://www.academia.edu/17980449/.


\bibitem{godsil}
Godsil, C.  (1985). Inverses of trees. Combinatorica, 5, 33–39.


\bibitem{grassl}
Grassl, M. Bounds on the minimum distance of linear codes and quantum codes. \url{https://www.codetables.de}



\bibitem{Ho}
Ho, D., \& Johnsen, T. (2026). On the hull-variation problem of equivalent vector rank metric codes. Advances in Mathematics of Communications, 22, 163--174.

\bibitem{ho2026}
Ho, D. (2026). Hermitian hull-variation of vector rank-metric codes and self-orthogonal generalized Gabidulin codes. arXiv preprint arXiv:2605.20109.



\bibitem{KKKS06}
Ketkar, A., Klappenecker, A., Kumar, S., \& Sarvepalli, P. K. (2006). Nonbinary stabilizer codes over finite fields. IEEE Transactions on Information Theory, 52(11), 4892--4914.


\bibitem{li2024}
Li, Y., \& Zhu, S. (2024). On symplectic hulls of linear codes and related applications. Journal of Applied Mathematics and Computing, 70(3), 2603-2622.


\bibitem{Lovasz79} 
Lovász, L. (1979). On determinants, matchings and random algorithms. Fundamentals of Computation
Theory, 79, 565–574.


\bibitem{Lovasz80}
Lovász, L.  (1980). Matroid matching and some applications.
Journal of Combinatorial Theory, Series B, 28, 208--236.

\bibitem{Lovasz89}
Lovász, L. (1989). Singular spaces of matrices and their application in combinatorics. Bulletin of the Brazilian Mathematical Society, 20(1), 87–99.



\bibitem{Ling}
Luo, G., Ezerman, M.F., Grassl, M., \& Ling, S. (2024). Constructing quantum error-correcting codes that require a variable amount of entanglement.  Quantum Information Processing, 23(4). 


\bibitem{Massey}
Massey, J.L. (1992). Linear codes with complementary duals. 
Discrete Mathematics, 106, 337--342.

\bibitem{nielsenbook}
Nielsen, M. A., \& Chuang, I. L. (2010). Quantum Computation and Quantum Information, 10th Anniversary Edition. Cambridge University Press. 




\bibitem{Pllaha}
Pllaha, T. (2020). Symplectic isometries of stabilizer codes. 
Journal of Algebra and Its Applications, Vol. 19(2), 2050021. 



\bibitem{Rains99}
Rains, E. M. (1999). Nonbinary quantum codes. IEEE Transactions on Information Theory, 45(6), 1827--1832.


\bibitem{Sendrier97}
Sendrier, N. (1997). On the dimension of the hull. SIAM Journal on Discrete Mathematics, 10(2), 282--293.  

\bibitem{Sendrier2000}
Sendrier, N. (2000). Finding the permutation between equivalent linear codes: The support splitting algorithm. IEEE Transactions on Information Theory, 46(4), 1193--1203.

\bibitem{Tonchev} 
Tonchev, V. D. (2002). Error-correcting codes from graphs, Discrete Math. 257 (2--3), 549--557.

\bibitem{Tutte}
Tutte, W. T. (1947). The factorization of linear graphs. Journal of the London Mathematical
Society. 22(2):107–111. 

\bibitem{XuDu21}
Xu, H. Q., \& Du, W. (2021). Constructions of symplectic LCD MDS codes. Bulletin of the Malaysian Mathematical Sciences Society, 44, 3377--3390.

\bibitem{XuDu22}
Xu, H. Q., \& Du, W. (2022). On some binary symplectic self-orthogonal codes. Applicable Algebra in Engineering, Communication and Computing, 33, 321--337.


\bibitem{li2024b}
Zhu, S., Li, Y., \& Li, S. (2024). Symplectic self-orthogonal and linear complementary dual codes from the Plotkin sum construction. Finite Fields and Their Applications, 96, 102425.
\end{thebibliography}
\end{document}